\documentclass[12pt,a4paper,reqno]{amsart}
\usepackage[margin=2cm]{geometry}
\usepackage{amssymb}
\usepackage{amsmath,amsthm}
\usepackage{amsfonts}
\usepackage{mathrsfs}
\usepackage{color}
\usepackage{cite}
\usepackage{enumitem,xcolor}
\usepackage{hyperref,cleveref}

\title[On a Perturbed Critical p-Kirchhoff-Type Problem]{On a Perturbed Critical p-Kirchhoff-Type Problem}

\author{G. N. Cunha}
\address[G. N. Cunha]{Instituto de Matem\'{a}tica e Estat\'{i}stica, Universidade Federal de Goi\'{a}s, Goi\^{a}nia GO74001-970, Brazil}
\email{gabriel.neves@discente.ufg.br}

\author{F. Faraci}
\address[F. Faraci]{Department of Mathematics and Computer Sciences, University of Catania, 95125 Catania, Italy}
\email{ffaraci@dmi.unict.it}

\author{K. Silva}
\address[K. Silva]{Instituto de Matem\'{a}tica e Estat\'{i}stica, Universidade Federal de Goi\'{a}s, Goi\^{a}nia GO74001-970, Brazil}
\email{kayesilva@ufg.br}
\thanks{}

\newcommand{\R}{{\mathbb R}}

\newcommand{\ds}{\displaystyle}
\newtheorem{theor}{Theorem}[section]

\newtheorem{pro}{Proposition}[section]
\newtheorem{lem}{Lemma}[section]
\newtheorem{cor}{Corollary}[section]
\newtheorem{rem}{Remark}[section]

\usepackage{mwe}
\usepackage{comment}

\begin{document}
	\begin{abstract}
	In this paper we deal with a stationary non-degenerate $p-$Kirchhoff type problem with critical non-linearity and a subcritical parametrized perturbation. We work  on bounded domains of the Euclidean space, without any restriction on the dimension or on $p>0$. Variational methods will be utilized in combination with an analysis of the fibering functions of the energy functional, in order to obtain ground state solutions, as well as Mountain Pass solutions, depending on the values of the parameter. A local analysis of the energy functional will allow us to obtain non-trivial solutions even beyond the extremal parameter.
		
	\end{abstract}
	
	\maketitle
	\begin{center}
		\begin{minipage}{12cm}
			\tableofcontents
		\end{minipage}
	\end{center}
	\smallskip
	
	
	\smallskip
	
	\section{Introduction}\label{Intro}
	
	In this paper we deal with the following stationary Kirchhoff type problem:	
	\begin{equation}\label{P}
		\left\{\begin{array}{lr}
			-M\left(\ds{\int_{\Omega}|\nabla u|^p}dx\right)\Delta_pu=|u|^{p^{\star}-2}u+\lambda f(x,u)&in  \ \ \ \Omega \\
			u=0&on \ \partial\Omega,
		\end{array}\right.
	\end{equation}
	where $1<p < +\infty $, $\Omega$ is a bounded domain in $\R^N$ with smooth boundary,  $p^{\star}=\frac{pN}{N-p}$, $M:[0,+\infty[\mapsto [0,+\infty[$ is a continuous function with $\hat{M}(t):=\ds{\int_{0}^{t}M(s)ds}$, $f:\Omega \times \mathbb{R}\mapsto \mathbb{R}$ is a Carath\'eodory function with primitive $F(x,t)=\ds{\int_{0}^{t}f(x,s)ds}$ for each $x \in \Omega$, $t\in\R$.

According to the survey \cite{pucci2019progress}, Kirchhoff problems arise from the study of the transverse oscillations of a stretched string. The original  equation is
\begin{equation}\label{kirchhoff}
	\rho h u_{tt}-\left\{\rho_0+\frac{E h}{2L}\int_{0}^{L}|u_x|^2dx\right\}u_{xx}+\delta u_x+f(x,u)=0,
\end{equation}
where   $u = u(t, x)$ is the lateral displacement at the time $t$ and at the space
coordinate $x$, $E$ the Young modulus, $\rho$ the mass density, $h$ the cross section area, $L$ the length of the string,
$\rho_0$ the initial axial tension, $\delta$ the resistance modulus, and $f$ is the external force. When $\delta=f=0$, \eqref{kirchhoff} was
introduced by Kirchhoff in \cite{kirchhoff1897vorlesungen}. Further details and a study of the physical phenomena described by Kirchhoff’s classical
theory can be found in \cite{villaggio1997mathematical}.
	In the last years, the existence and multiplicity of solutions for Kirchhoff problems with a critical non-linearity have received considerable attention. As a matter of fact, the main difficulty in dealing with such problems is the  lack of
	compactness of the Sobolev embedding $W_0^{1,p}(\Omega)\subset L^{p^\star}(\Omega)$, which prevents the application of
	standard variational methods.
	
	The existence and multiplicity of solutions of Kirchhoff type equations with critical exponents have been investigated by using different techniques, such as truncation and variational
	methods, the Nehari manifold approach, the Ljusternik-Schnirelmann category theory, genus
	theory (see for instance \cite{julio2006elliptic}--\cite{figueiredo2012multiplicity} and the references therein).

	In the present paper, we apply an idea introduced in \cite{FS}, which was inspired by the fibering method  in \cite{pokhozhaev1990fibration} and the  notion of extremal parameters described in \cite{Y}, to analize the topological changes occured on the energy functional as the parameter $\lambda$ varies. With such a technique, we do not need  to consider the second order derivative of the fiber function, except for the non-existence results. More specifically  we employ the Second Lions's Concentration Compactness principle to obtain Mountain Pass type solutions (see \cite{alves2010class},
\cite{figueiredo2013existence},\cite{hebey2015compactness},\cite{hebey2016multiplicity},\cite{naimen2014positive},\cite{naimen2019two},\cite{yao2016multiplicity}), 
but also to establish the sequential weak lower semi-continuity of the energy functional, with a proof inspired by \cite{EM}, which is, in turn (as far as we know), the only work so far where this approach has been utilized. Also, in the cases beyond the extremal parameter, we prove it is still possible to obtain non-trivial solutions, as long as we minimize the energy functional locally, as in \cite{FS}.
	
	We are going to look for solutions of problem \eqref{P} in the Sobolev space $W^{1,p}_0(\Omega)$. This linear space is endowed with the norm
	\begin{equation*}
		 \| u \|:= \left(\int_{\Omega}|\nabla u|^p dx\right)^{\frac{1}{p}}
	\end{equation*}
and continuously embedded  into $L^{p^\star}(\Omega)$	with embedding constant
\begin{equation}\label{azzz}
	\ds{S=\sup_{u \in W_0^{1,p}(\Omega)\setminus \{0\}}\frac{\|u\|_{p^\star}^{p^\star}}{\|u\|^{p^\star}}}.
\end{equation}

	A weak solution for problem \eqref{P} is a critical point of the energy functional $\Phi_{\lambda}:W_0^{1,p}(\Omega)\mapsto \R$ given by

     \begin{equation*}
		\ds{\Phi_\lambda(u)= \frac{1}{p}\hat{M} (\| u \|^p)-\frac{1}{p^\star}\| u \|^{p^\star}_{p^\star}-\lambda\int_{\Omega}F(x,u(x))dx.}
	\end{equation*}

In order to control the behavior of the fibers at $0$ and $+\infty$, as well as to establish the coercivity of the energy functional, we need the following hypotheses on the non-local term $M$:
	\begin{enumerate}[label=\textcolor{blue}{$(\rho_\arabic*)$:}, ref=$\rho_\arabic*$]
	
		\item\label{rho1} $\ds{\lim_{t \to 0^{+}}M(t) > 0}$;
	
		\item \label{rho2}$\ds{\lim_{t \to +\infty}\frac{M(t)}{t^\frac{r-p}{p}}>0}$ for some  $r>p^\star$.
	\end{enumerate}
	The sequential weak lower semi-continuity of the energy functional, on the other hand, is associated with the following conditions:
	\begin{enumerate}[label=\textcolor{blue}{$(\beta_\arabic*)$:}, ref=$\beta_\arabic*$]
		\item\label{beta1} $\ds{\inf_{t > 0}\frac{\hat{M}(t)}{t^\frac{p^\star}{p}} \geq
		S\frac{p}{p^\star}}$;
		
		\item\label{beta2} $\ds{\hat{M}(t+s)\geq \hat{M}(t)+\hat{M}(s)}$  for all $t >0$ and $s>0$.
	\end{enumerate}
Notice that conditions $\eqref{beta1}$ and $\eqref{beta2}$ imply that $\hat M$ is strictly increasing.

 The existence of a Mountain Pass solution comes mainly from the next condition. This hypothesis (which is stronger than \eqref{beta1}) is also related to the non-existence results:

    \begin{enumerate}[label=\textcolor{blue}{$(\gamma_\arabic*)$:}, ref=$\gamma_\arabic*$]
	\item\label{beta3} $\ds{\inf_{t > 0}\frac{M(t)}{t^{\frac{p^\star}{p}-1}} > S}$
    \end{enumerate}	
	
An example of a function satisfying conditions  \eqref{rho1}, \eqref{rho2}, \eqref{beta1}, \eqref{beta2}, and \eqref{beta3}, is
	\begin{equation}\label{example}
		M(t):=a+bt^{\alpha-1},
	\end{equation}

for suitable values of $\alpha>1$ and $a,b>0$:
\begin{itemize}
	\item[(i)]
	Assume $\alpha> \frac{N}{N-p}$; then
	\begin{equation*}
		\ds{\inf_{t>0}\frac{\hat{M}(t)}{t^{\frac{p^\star}{p}}}} \geq \frac{p}{p^\star}S
	\end{equation*}
	if, and only if,

	\begin{equation*}
		a^{\frac{N(\alpha-1)-\alpha p}{p}}b\geq \left[\frac{p}{p^\star}S\frac{N(\alpha-1)-p\alpha}{N(\alpha-1)-p\alpha+p}\right]^{\frac{(N-p)(\alpha-1)}{p}}\alpha\frac{p}{N(\alpha-1)-p\alpha}.
	\end{equation*}
	
	\item[(ii)]
	
	Also, for $\alpha> \frac{N}{N-p}$, the following assertion holds true:
	\begin{equation*}
		\ds{\inf_{t>0}\frac{M(t)}{t^{\frac{p^\star}{p}-1}}} >S
	\end{equation*}
	if, and only if,

	\begin{equation*}
		a^{\frac{N(\alpha-1)-\alpha p}{p}}b> \left[S\frac{N(\alpha-1)-p\alpha}{N(\alpha-1)-p\alpha+p}\right]^{\frac{(N-p)(\alpha-1)}{p}}\frac{p}{N(\alpha-1)-p\alpha}.
	\end{equation*}
	
	\item[(iii)] For $\alpha>1$, $\ds{\lim_{t \to 0}M(t)>0}$.
	
	\item[(iv)] For any $r>p^\star$ such that $\alpha p>r$, there holds
	\begin{equation*}
		\ds{\lim_{t \to +\infty} \frac{M(t)}{t^{\frac{r-p}{p}}}>0}.
	\end{equation*}

	\item[(v)] For $\alpha>1$, the inequality
	\begin{equation*}
		\hat{M}(t+s)\geq \hat{M}(t)+\hat{M}(s) \ \ \ \ \forall \ t,s\ \in [0,+\infty[
	\end{equation*}
	holds true.

	\item[(vi)] For $\alpha \in (1,\frac{N}{N-p}]$, there holds
	\begin{equation*}
		\ds{\inf_{t>0}\frac{\hat{M}(t)}{t^{\frac{p^\star}{p}}}}=\left\{\begin{array}{lr}
			0& {\rm if} \ \   \alpha<\frac{N}{N-p} \\
			\frac{b}{\alpha}& {\rm if} \ \  \alpha=\frac{N}{N-p}.		\end{array}\right. \end{equation*}
\end{itemize}

Some comparison with related literature is in order.  For $S=S_N^{-\frac{p^\star}{p}}$ in \eqref{azzz}, and  $\alpha=p=2$ in \eqref{example}, we obtain the problem studied in the work \cite{FS}; in that paper, the conditions corresponding to \eqref{beta1} and \eqref{beta3} would be, respectively,
		\begin{equation*}
			a^{\frac{N-4}{2}}b\geq \frac{(N-4)^{\frac{N-4}{2}}}{N^{\frac{N-2}{2}}S_N^{\frac{N}{2}}}4
		\end{equation*}
	and
		\begin{equation*}
			a^{\frac{N-4}{2}}b> \frac{(N-4)^{\frac{N-4}{2}}}{(N-2)^{\frac{N-2}{2}}S_N^{\frac{N}{2}}}2.
		\end{equation*}
	In the present paper  we also achieve  an improvement with respect to \cite{faraci2020critical} concerning the semicontinuity property: in that paper, the condition corresponding to \eqref{beta1}, i.e.,
			\begin{equation*}
				\ds{\inf_{l > 0}\frac{\hat{M}(l)}{l^\frac{p^\star}{p}} \geq
					c_p}
			\end{equation*}
			where
			\begin{equation*}
				c_p=\left\{\begin{array}{lcl}
					\left(2^{p-1}-1\right)^{\frac{p^\star}{p}}\frac{p}{p^\star}S_N^{-\frac{P^\star}{p}}&if&p\geq2,\\
					2^{2p^\star-1-\frac{p^\star}{p}}\frac{p}{p^\star}S_N^{-\frac{p^\star}{p}}&if&1<p<2,
				\end{array}\right.
			\end{equation*} is  more restrictive than $(\beta_1)$ since  for $p\neq 2$ there holds
		\begin{equation*}
			 c_p>\frac{p}{p^\star}S_N^{-\frac{p^\star}{p}}.
		\end{equation*}
		
		\smallskip
		
	The following hypotheses on  the perturbation $f$ will be used throughout this work.
	\begin{enumerate} [label=\textcolor{blue}{$(f_\arabic*)$:}, ref=$f_\arabic*$]
		\item\label{f1} There exist $c_1,c_2>0$ and $q \in (p,p^\star)$ such that $ |f(x,t)|\leq c_1+ c_2|t|^{q-1}$ for $t \in \R$, and a.e. in $\Omega$;
		\item\label{f2} $\ds{\lim_{t \to 0}} \frac{f(x,t)}{|t|^{p- 1}}=0$ uniformly on $x \in \Omega$;
		\item \label{f3} $f(x,t)>0$ for every $t>0$, a.e in $\Omega$; and $f(x,t)<0$ for every $t<0$, a.e in $\Omega$. Moreover there exists $\mu>0$ such that  $f(x,t)\geq \mu >0$ for a.a. $x\in \Omega$ and every $t\in I$, being $I$ an open interval of $(0,+\infty)$.
	\end{enumerate}
An example of a function satisfying the conditions \eqref{f1}, \eqref{f2}, and \eqref{f3}, is
\begin{equation*}
	\begin{array}{rcl}
			f(x,t)=|t|^{q-2}t&\forall& t \in \R,
	\end{array}
\end{equation*}
where $q \in (p,p^\star)$ is fixed.

Now we introduce the main results of this paper, which we prove in the next sections.

The first result justifies the necessity for the parameterized perturbation.
\begin{theor}\label{main0}
	Under condition \eqref{beta3} there exists a number  $\lambda_1^\star>0$ such that for all $-\infty<\lambda< \lambda_1^{\star}$ problem \eqref{P} possesses only the trivial solution.
\end{theor}
 The following existence result is the main goal of this paper.
\begin{theor}\label{main1}
	Assume conditions \eqref{rho1}, \eqref{rho2}, \eqref{beta1}, \eqref{beta2}, \eqref{f1}, \eqref{f2}, and \eqref{f3}. Then, there exists $\lambda_0^\star\geq0$ such that 
	\begin{itemize}	
				\item[(i)] if $\lambda>\lambda_0^\star$, then the energy functional $\Phi_{\lambda}$ has a global minimizer $u_\lambda$ such that $I_\lambda=\Phi_{\lambda}(u_\lambda)<0$ ( in particular that $u_\lambda\neq 0$);
		\item[(ii)] if $\lambda=\lambda_0^\star$, then the energy functional $\Phi_{\lambda}$ has a global minimizer $u_{\lambda_0^\star}$ such that $I_{\lambda_0^\star}=0$;  if the inequality in condition \eqref{beta1} is strict, then  $u_{\lambda_0^\star}\neq 0$; 
   	\item[(iii)]   if $\lambda<  \lambda_0^\star$, then for all $u \in W_0^{1,p}(\Omega)\backslash\{0\}$ there holds $\Phi_{\lambda}(u)>0$. Therefore, $u_\lambda=0$ is the only global minimizer of $\Phi_\lambda.$
	\end{itemize}
	\end{theor}

The next result shows that although we may not find non-trivial global minimizers for the energy functional for $\lambda<\lambda_0^\star$, we may still find non-trivial local minimizers, as long as the parameter $\lambda$ is close enough to $\lambda_0^\star$.
		\begin{theor}\label{main2}
			Assume conditions \eqref{rho1}, \eqref{rho2}, \eqref{beta1}, \eqref{beta2}, \eqref{f1}, \eqref{f2}, and \eqref{f3}. If the inequality in condition \eqref{beta1} is strict, then there exists $\epsilon >0$ small enough so that for each $\lambda \in (\lambda_0^\star-\epsilon,\lambda_0^\star)$ the energy functional $\Phi_{\lambda}$ possesses a local minimizer with positive energy.
	\end{theor}
The next result states the existence of a mountain pass type solution.
\begin{theor}\label{main3}
Assume conditions \eqref{rho1}, \eqref{rho2},  \eqref{beta2}, \eqref{beta3}, \eqref{f1}, \eqref{f2}, and \eqref{f3}.  Then, there exists  $\epsilon>0$ small enough such that for each $\lambda>\lambda_0^\star-\epsilon$, problem \eqref{P} has a solution of mountain pass type.
 \end{theor}

	\section{Abstract Results}
	Now we proceed to describe the abstract results which allow us to deduce our main theorems, stated in  Section \ref{Intro}.

	For each $u \in W_0^{1,p}(\Omega)\backslash\{0\}$ and  $\lambda \geq 0$, we define the fiber function $\psi_{\lambda,u}:[0,+\infty[\mapsto \R$ by $\psi_{\lambda,u}(t):=\Phi_\lambda(tu)$.
	
\begin{lem}\label{fibersfirstproperties} Assume conditions \eqref{f1} and \eqref{f2}. Then, the following assertions  hold true:
		
\begin{itemize}	
\item[(i):]	Under condition \eqref{rho1}, there exist $\epsilon_1=\epsilon_1(\lambda,u)>0$ and $\epsilon_2=\epsilon_2(\lambda,u)>0$ such that $\psi_{\lambda,u}(t)>0 \ \forall t \in (0,\epsilon_1)$, and $\psi'_{\lambda,u}(t)>0 \ \forall t \in (0,\epsilon_2)$;
\item[(ii):] Under condition \eqref{rho2}, there holds $\ds{\lim_{t\to \infty}\psi_{\lambda,u}(t)}=+\infty$ and $ \ds{\lim_{t\to \infty}\psi'_{\lambda,u}(t)=+\infty}$.
\end{itemize}
		\end{lem}
	\begin{proof} We will prove the claim for $\psi_{\lambda,u}$. In fact, write
			\begin{equation*}\label{fiber1}\psi_{\lambda,u}(t)=t^p\|u\|^p\left[\frac{1}{p}\frac{\hat{M}(t^p\|u\|^p)}{t^p\|u\|^p}-\frac{t^{p^\star-p}}{p^\star}\frac{\|u\|_{p^\star}^{p^\star}}{\|u\|^p}-\frac{\lambda}{\|u\|^p}\int_{\Omega}\frac{F(x,tu(x))}{t^p}dx\right].
			\end{equation*} 		
Note that by conditions \eqref{f1} and \eqref{f2},  for each $\varepsilon>0$ there exists $c>0$ such that
\[|F(x,t)|\leq \varepsilon |t|^p+c|t|^q  \ \hbox{for all $t\in\R$, a.e. in $\Omega$. }\]
Therefore,  \begin{equation*}\lim_{t\to 0}\int_{\Omega}\frac{F(x,tu(x))}{t^p}dx=0.
\end{equation*}
By assumption \eqref{rho1} and  De l'Hospital rule 
\begin{equation*}
	\ds{\lim_{t\to 0}\frac{\hat{M}(t^p\|u\|^p)}{t^p\|u\|^p}>0},
\end{equation*}
and the first conclusion in $(i)$ follows.

By  $(\rho_2)$ and continuity of $M$, there exists positive constants $c_1, c_2$ such that  
\[\hat M(t)\geq c_1 t^{\frac{r}{p}} -c_2 \ \hbox{for all $t\geq 0$}. \]
Thus, from $(f_1)$, and possibly different constants $c_i$, 
\begin{align*}\psi_{\lambda,u}(t)&=\frac{1}{p}{\hat{M}(t^p\|u\|^p)}-\frac{t^{p^\star}}{p^\star}{\|u\|_{p^\star}^{p^\star}}-{\lambda}\int_{\Omega}{F(x,tu(x))}dx\\
&\geq \frac{1}{p}c_1 t^r\|u\|^r-\frac{t^{p^\star}}{p^\star}{\|u\|_{p^\star}^{p^\star}}-c_3t^q\|u\|_q^q -c_2
			\end{align*} 
and the first claim in $(ii)$ holds. 
\end{proof}

For each $u \in W_0^{1,p}(\Omega)\backslash\{0\}$, consider now the following system:
	\begin{equation}\label{systemlambdat}\left\{\begin{array}{l}
			\psi_{\lambda,u}(t)=0\\
			\psi'_{\lambda,u}(t)= 0\\
			\psi_{\lambda,u}(t)=\inf_{s > 0}\psi_{\lambda,u}(s).
		\end{array}\right.\end{equation}

	\begin{lem}\label{systemlambdatsolutionexistence} Assume conditions \eqref{f1}, \eqref{f2}, \eqref{rho1}, \eqref{rho2} and  \eqref{beta1}. Then, system \eqref{systemlambdat} has a solution $(\lambda_0(u),t_0(u))$ for each $u \in W_0^{1,p}(\Omega)\backslash\{0\}$. Furthermore, the solution is unique with respect to $\lambda$.
	\end{lem}
\begin{proof}
	We proceed by showing first that for each $u \in W_0^{1,p}(\Omega)\backslash\{0\}$, the set $$\Lambda_u:=\{\lambda\geq0 :  \  \ \inf_{t\geq 0}\psi_{\lambda,u}(t)\geq 0\}$$ is non empty.	
	Fix  $\bar{\lambda}>0$. By Lemma \ref{fibersfirstproperties}, there exist $\epsilon_1, \delta_1>0$ such that $\psi_{\bar{\lambda},u}>0$ in $(0,\epsilon_1)\cup (\delta_1,+\infty)$. 
	
	 For $\lambda\leq \bar{\lambda}$, since   $\psi_{\lambda,u}\geq \psi_{\bar{\lambda},u}$, the fiber $\psi_{\lambda,u}$ is positive over  $(0,\epsilon_1)\cup (\delta_1,+\infty)$.
	
	Take a positive monotone sequence $\{\lambda_k\}_{k \geq 1}$ converging to zero. The sequence of continuous real functions $\{\psi_{\lambda_k,u}\}_{ k \geq 1}$ converges uniformly to $\psi_{0,u}$ over the compact interval $[\epsilon_1,\delta_1]$. We observe that $\inf_{[\epsilon_1,\delta_1]}\psi_{0,u}>0$ as it follows  by condition \eqref{beta1} and by the fact that the constant $S$ in \eqref{azzz}
is not attained. Therefore, there exists a $k_0$ such that the fiber $\psi_{\lambda_k,u}$ is positive over the interval $[\epsilon_1,\delta_1]$ for all $k>k_0$. For $k$ large enough such that $\lambda_{k}< \bar{\lambda}$, we get that $\lambda_k \in \Lambda_u$. Also, $\Lambda_u$ is bounded from above, since for fixed $t_0>0$, $\lim_{\lambda\to\infty}\psi_{\lambda,u}(t_0)=-\infty$. 
	
	Now we define the canditate $\lambda_0(u):=\sup \Lambda_u$. For each positive $\epsilon\leq \epsilon_0$ (where $\epsilon_0$ is fixed), we  denote by $t_0(\epsilon)$, the first critical point of $\psi_{\lambda_0(u)+\epsilon,u}$ such that
\begin{equation}\label{olaa}
	\psi_{ \lambda_0(u)+\epsilon,u}(t_0(\epsilon))<0 \ .
\end{equation}
Since the function $\epsilon \mapsto\psi_{\lambda_0(u)+\epsilon,u}(t)$ is decreasing, the map $\epsilon \mapsto t_0(\epsilon)$ is bounded from below by the first root of the fiber $\psi_{\lambda_0(u)+\epsilon_0,u}$.  Define the candidate $t_0(u):=\liminf_{\epsilon \to 0}t_0(\epsilon)$. By taking the liminf on inequality \eqref{olaa} as $\epsilon$ goes to zero, we get $\psi_{\lambda_0(u),u}(t_0(u))\leq 0$. Since $\lambda_0(u)$ belongs to $\Lambda_u$ as well, there holds $\psi_{\lambda_0(u),u}(t)\geq0$ for all $t>0$, in particular for $t=t_0(u)$. Therefore, $\psi_{\lambda_0(u),u}(t_0(u))=0$

Let us now prove the uniqueness.

Assume that  the ordered pairs $(\lambda_0(u),t_0(u))$,$(\lambda^{'}_0(u),t^{'}_0(u))$ are both solutions of system \eqref{systemlambdat}. Assume without loss of generality that $\lambda_0^{'}(u)\geq\lambda_0(u)$. Then,  $0\leq\psi_{\lambda_0^{'}(u),u}(t_0(u))\leq\psi_{\lambda_0(u),u}(t_0(u))=0$, 
which implies $\lambda_0^{'}(u)=\lambda_0(u)$.
The proof is concluded.
\end{proof}

\begin{cor}\label{zerohomogeneous} For all $u \in W_0^{1,p}(\Omega)\backslash\{0\}$ and all $k \geq 0$, there holds $\lambda_0(ku)=\lambda_0(u)$.
\end{cor}		
\begin{proof} For $u \in W_0^{1,p}(\Omega)\backslash\{0\}$ and $k\geq 0$, there holds $ku \in W_0^{1,p}(\Omega)\backslash\{0\}$. System \eqref{systemlambdat} possesses a solution $(\lambda_0(ku),t_0(ku))$, where the first coordinate is unique:
	
	\begin{equation}\label{systemlambdathomogeneous}
		\left\{\begin{array}{l}
			\psi_{\lambda_0(ku),ku}(t_0(ku))=0\\
			\psi^{'}_{\lambda_0(ku),ku}(t_0(ku))= 0\\
			\psi_{\lambda_0(ku),ku}(t_0(ku))=\inf_{l > 0}\psi_{\lambda_0(ku),ku}(l).
		\end{array}\right.
	\end{equation}
System \eqref{systemlambdathomogeneous} may be rewritten as	
		\begin{equation*}\left\{\begin{array}{l}
			\psi_{\lambda_0(ku),u}(kt_0(ku))=0\\
			\psi^{'}_{\lambda_0(ku),u}(kt_0(ku))= 0\\
			\psi_{\lambda_0(ku),u}(kt_0(ku))=\inf_{l > 0}\psi_{\lambda_0(ku),u}(kl)=\inf_{l > 0}\psi_{\lambda_0(ku),u}(l).
		\end{array}\right.
	\end{equation*}
By uniqueness, we conclude that

	\begin{equation*}
	\lambda_0(ku)=	\lambda_0(u).
\end{equation*}\end{proof}
Now, we define an extremal parameter which will play an important role in our analysis:	
\begin{equation}\label{extremalparameterdefinition}
	\lambda^{\star}_0:=\inf_{u \in W_0^{1,p}(\Omega)\backslash\{0\}}\lambda_0(u).
\end{equation}

\begin{lem}\label{extremalparametersign} Assume conditions \eqref{f1}, \eqref{f2}, and \eqref{f3}. The following assertions hold true.
	\begin{itemize}
		\item[(i):] Under condition \eqref{beta1}, there holds $\lambda_0^\star\ge 0$.
		\item[(ii):] If the inequality in condition \eqref{beta1} is strict, then $\lambda_0^{\star}>0$ and viceversa. 
			\end{itemize}		
\end{lem}

	\begin{proof} $(i)$ Let us perform a proof by contradiction. Assume there exists a sequence $\{u_k\}_{k\geq 1} \in W_0^{1,p}(\Omega)\backslash\{0\}$ such that (we use the notation $\lambda_k:=\lambda_0(u_k)$)
		\begin{equation}\label{limitlambdak}
			-\infty\leq\lim_{k \to \infty}\lambda_k= \lambda_0^\star< 0  .
		\end{equation}
		We may assume by Corollary \eqref{zerohomogeneous} that $\|u_k\|=1$ for all $k \geq 1$. By the definition of $\lambda_0(u_k)$, there exists a sequence $\{t_k\}_{k \geq 1}$ of positive numbers (which is bounded, according to Lemma \eqref{fibersfirstproperties} item $(ii)$ ) 
		such that
		
		\begin{equation*}
			\frac{1}{p}\hat{M}(t_k^p\|u_k\|^p)-\frac{1}{p^{\star}}t_k^{p^\star}\|u_k\|^{p^\star}_{p^\star}=\lambda_k\int_{\Omega}F(x,t_ku_k)dx \ \ \forall \ k\geq 1.
		\end{equation*}
	By the Sobolev embbeding,
		\begin{equation}\label{inequalityt_k}
			\frac{1}{p}\hat{M}(t_k^p)-\frac{1}{p^{\star}}t_k^{p^\star}S\leq \lambda_k\int_{\Omega}F(x,t_ku_k)dx \ \ \forall \ k\geq 1.
		\end{equation}
The contradiction follows from \eqref{inequalityt_k} and \eqref{limitlambdak}. Therefore, there holds $\lambda_0^\star\geq0$.

$(ii)$ Let $L>S\frac{p}{p^\star}$ such that 
$\inf_{t > 0}\frac{\hat{M}(t)}{t^\frac{p^\star}{p}} \geq L.$  
Arguing as in the proof of item $(i)$, assume by contradiction that $\ds{\lim_{k\to +\infty}\lambda_k=\lambda_0^\star=0}$, where $\lambda_k=\lambda_0(u_k)$ with $\|u_k\|=1$. Thus, there exists a  sequence $\{t_k\}_{k\geq 1}$ of positive numbers such that 
	 	\begin{equation*}
\left(L-S\frac{p}{p^\star}\right)t_k^{p^\star}	\leq {\hat M(t_k^p})-S\frac{p}{p^\star}t_k^{p^\star}	\le \lambda_kp\int_\Omega{F(x,t_k u_k)}dx.
	\end{equation*}
	
The right hand side tends to zero since $\lambda_k\to 0$ and $\left\{\int_\Omega F(x, t_k u_k) dx\right\}$ is bounded due to  the growth of $F$, and the fact that $\{t_k\}$ and $\{u_k\}$ are bounded. This implies that $\ds\lim_{k\to +\infty}t_k=0$.  Dividing the previous inequality by $t_k^p$, we get 
	\begin{equation*}
	\frac{\hat M(t_k^p)}{t_k^p}-S\frac{p}{p^\star}t_k^{p^\star-p}\le {\lambda_k}p\int_\Omega\frac{F(x,t_k u_k)}{t_k^p}dx,
	\end{equation*} which contradicts assumption \eqref{rho1} when passing to the limit  as $k\to\infty$.
Therefore, $\lambda_0^\star>0$.

\smallskip

Let us prove the viceversa. Assume that condition \eqref{beta1} holds with equality. We will prove that $\lambda_0^\star=0$. 
Without loss of generality we may assume that $0\in \Omega$. Fix  a ball of radius $r>0$ such that $B_{2r}(0)\subset \Omega$ and let $\varphi$ a function in $C^\infty_0(B_{2r}(0))$ such that $\varphi(x)=1$ in $B_{r}(0)$, $0\leq \varphi\leq 1$ and $|\nabla \varphi|\leq 2$.

Put 
\begin{equation*}\label{uepsilon}
v_\varepsilon(x)=\frac{\varphi(x)}{\left(\varepsilon+|x|^\frac{p}{p-1}\right)^{\frac{N-p}{p}}}  \qquad 
\hbox{and} \qquad
u_\varepsilon(x)=\frac{v_\varepsilon(x)}{\|v_\varepsilon\|}.
\end{equation*}

By \cite{GV} (see also \cite{BN}), there exists a constant $K=K(N,p)$ such that 

$$\ds\|v_\varepsilon\|^p=K\varepsilon^{-\frac{N-p}{p}}+O(1); \qquad \ds\|v_\varepsilon\|_{p^\star}^p=KS^{\frac{p}{p^\star}} \varepsilon^{-\frac{N-p}{p}}+O(\varepsilon); $$
 $$\ds\|u_\varepsilon\|=1; \qquad \ds\|u_\varepsilon\|_{p^\star}^{p^\star}=S+O(\varepsilon^{\frac{N}{p}}).$$
 We deduce in particular that $$\ds \|v_\varepsilon\|= K^{\frac{1}{p}}\varepsilon^{-\frac{N-p}{p^2}}+ O(\varepsilon^{\frac{(N-p)(p-1)}{p^2}}).$$
 
Let $t_0>0$ such that  $$\inf_{t > 0}\frac{\hat{M}(t)}{t^\frac{p^\star}{p}}=\frac{\hat{M}(t_0^p)}{t_0^{p^\star}}=S\frac{p}{p^\star}.$$ 
Then, \begin{eqnarray*}
	\psi_{\lambda,u_\varepsilon}(t_0)&= &	\frac{1}{p}\hat{M}(t_0^p)-\frac{1}{p^{\star}}t_0^{p^\star}\|u_\varepsilon\|^{p^\star}_{p^\star}-\lambda\int_{\Omega}F(x,t_0u_\varepsilon)dx \\
	&=&-\frac{1}{p^\star}t_0^{p^\star}O(\varepsilon^{\frac{N}{p}})-\lambda \int_\Omega F(x,t_0u_\varepsilon)dx.
	\end{eqnarray*}
Let us estimate $\ds\int_\Omega F(x,t_0u_\varepsilon)dx$ from below.
 By assumption \eqref{f3}, one has that $f(x,s)\geq \mu \chi_I(s)$ (being $\chi_I$ the characteristic function of the interval $I$), so there exist $\alpha, \beta >0$ such that $F(x,s)\geq \tilde F(s):=\mu\int_0^s \chi_I(t) dt \geq \beta $ for every $s\geq \alpha$.  Following Corollary 2.1 of \cite{BN} and using the positivity and  monotonicity of $F$, 
\begin{align*}
\int_{\Omega}F(x,t_0 u_\varepsilon)dx&\geq \int_{|x|\leq r}F(x,t_0 u_\varepsilon) dx \geq \int_{|x|\leq r}F\left(x, \frac{t_0}{\|v_\varepsilon\|(\varepsilon+|x|^\frac{p}{p-1})^{\frac{N-p}{p}}}\right)dx \\&\geq \int_{|x|\leq r}\tilde F\left( \frac{t_0}{\|v_\varepsilon\|(\varepsilon+|x|^\frac{p}{p-1})^{\frac{N-p}{p}}}\right)dx \\&=c_1\varepsilon^\frac{N(p-1)}{p}\int_0^{r\varepsilon^{-\frac{p-1}{p}} }\tilde F\left( \frac{t_0}{\|v_\varepsilon\|}\left(\frac{\varepsilon^{-1}}{1+s^{\frac{p}{p-1}}}\right)^{\frac{N-p}{p}}\right)s^{N-1}ds
\end{align*}
One has that 
\begin{equation}\label{MM}
\tilde{F}\left( \frac{t_0}{\|v_\varepsilon\|}\left(\frac{\varepsilon^{-1}}{1+s^{\frac{p}{p-1}}}\right)^{\frac{N-p}{p}}\right)\geq \beta \hbox{ if $s$ is such that} \  \frac{t_0}{\|v_\varepsilon\|}\left(\frac{\varepsilon^{-1}}{1+s^{\frac{p}{p-1}}}\right)^{\frac{N-p}{p}}\geq\alpha.
\end{equation}
Notice that the second inequality of \eqref{MM} is equivalent to 
\begin{equation*}
\frac{t_0\varepsilon^{-\frac{(N-p)(p-1)}{p^2}}}{(K^{\frac{1}{p}}+O(\varepsilon^{\frac{N-p}{p}}))(1+s^{\frac{p}{p-1}})^{\frac{N-p}{p}}}\ge \alpha.
\end{equation*}
 Now, fix $c_2<\frac{t_0}{\alpha}$. If $s\leq c_2 \varepsilon^{-\frac{(p-1)^2}{p^2}}$, for $\varepsilon$ small enough, the above inequality holds true.   
 This implies that for an eventually smaller $r$,
 \begin{align*}
\int_{\Omega}F(x,t_0 u_\varepsilon)dx\geq c_3 \varepsilon^\frac{N(p-1)}{p}\int_0^{r\varepsilon^{-\frac{(p-1)^2}{p^2}} }\beta s^{N-1}ds =c_4\varepsilon^{\frac{N(p-1)}{p^2}},
\end{align*} for some positive constant $c_4$.
Hence, 
$$\psi_{\lambda,u_\varepsilon}(t_0)\leq -\frac{1}{p^\star}t_0^{p^\star}O(\varepsilon^{\frac{N}{p}})-\lambda c_4\varepsilon^{\frac{N(p-1)}{p^2}}=
\varepsilon^{\frac{N(p-1)}{p^2}}\left(O(\varepsilon^{\frac{N}{p^2}})-\lambda c_4\right)<0,
$$ for small $\varepsilon>0$. We  deduce that $\lambda_0(u_\varepsilon)<\lambda$ and because of the arbitrariness of $\lambda,$ $\lambda_0^\star=0$.
	\end{proof}

Now we  prove a continuity result, which will be useful for proving the existence of a minimizer of our problem. The proof we present here was inspired by \cite[Theorem 2.1]{EM}.
	\begin{lem}\label{semicontinuidade}Assume conditions \eqref{beta1} and \eqref{beta2}. Then, for all $\{u_k\}_{k\geq 1}\subset W_0^{1,p}(\Omega)$ such that $u_k \rightharpoonup u \in W_0^{1,p}(\Omega)$, and $\{\lambda_k\}_{k \geq 1}\subset \R$ such that $\lambda_k \rightarrow \lambda \in \R$,
		\begin{equation*}
			\Phi_\lambda(u) \leq \liminf_{k \to \infty}\Phi_{\lambda_k}(u_k).
		\end{equation*}	
	\end{lem}
	\begin{proof} Let $\{u_k\}_{k \geq 1} \subset W_0^{1,p}(\Omega)$ be such that $u_k \rightharpoonup u \in W_0^{1,p}(\Omega)$. Let's call $\liminf_{k \to \infty}\Phi_{\lambda_k}(u_k)=L$. By the second  Concentration-Compactness Lemma of Lions, there exist an at most countable index set $J$,
	 		a set of points $\{x_{j}\}_{j\in J}\subset\overline\Omega$ and two families of positive
	 		numbers $\{\eta_{j}\}_{j\in J}$, $\{\nu_{j}\}_{j\in J}$ such that
	 		\begin{align*}
	 		|\nabla u_{k}|^{p} & \rightharpoonup d\eta\geq|\nabla u|^{p}+\sum_{j\in J}\mu_{j}\delta_{x_{j}},\\
	 		|u_{k}|^{p^*} & \rightharpoonup d\nu=|u|^{p^*}+\sum_{j\in J}\nu_{j}\delta_{x_{j}},
	 		\end{align*}
	 		(weak star convergence in the sense of measures), where $\delta_{x_{j}}$ is the Dirac mass concentrated at
	 		$x_{j}$ and such that
	 		$$		S^{-\frac{p}{p^\star}}  \nu_{j}^{\frac{p}{p^*}}\leq\mu_{j} \qquad \mbox{for every $j\in J$}.$$
			Thus, we deduce 				
		\begin{eqnarray}\label{CC}
				L&=&\frac{1}{p}\hat{M}\left(\liminf_{k \to \infty}\int_{\Omega}|\nabla u_k|^pdx\right)-\frac{1}{p^\star}\left(\liminf_{k \to \infty}\int_{\Omega}|u_k|^{p^\star}dx \right)\nonumber \\
				&&-\liminf_{k \to \infty}\lambda_k\int_{\Omega}F(x,u_k(x))dx \nonumber\\
				&\overset{\eqref{beta2}}{\geq}& \frac{1}{p}\hat{M}\left(\int_{\Omega}|\nabla u|^pdx+\sum_{j \in J}\mu_j\right)-\frac{1}{p^\star}\left(\int_{\Omega}|u|^{p^\star}dx+\sum_{j \in J}\nu_j\right)-\lambda\int_{\Omega}F(x,u(x))dx \nonumber\\
				&\overset{\eqref{beta2}}{\geq}&\Phi_\lambda(u)+\frac{1}{p}\sum_{j \in J}\hat{M}\left(\mu_j\right)-\frac{1}{p^\star}\sum_{j \in J}\nu_j\nonumber\\
			&\overset{\eqref{beta1}}{\geq}& \Phi_\lambda(u)+\frac{S}{p^\star}\sum_{j \in J}\mu_j^{\frac{p^\star}{p}}-\frac{1}{p^\star}\sum_{j \in J}\nu_j\\
			&=& \Phi_\lambda(u).\nonumber
		\end{eqnarray}
	\end{proof}

	\begin{cor}\label{corolariosemicontinuidade}
		Assume that condition \eqref{beta2} holds true and the inequality in condition \eqref{beta1} is strict. Let  $\{u_k\}_{k \geq 1} \subset W_0^{1,p}(\Omega)$ be a sequence such that $u_k\rightharpoonup u \in W_0^{1,p}(\Omega)$ and  $\ds{\Phi_\lambda(u)=\lim_{k \to \infty}\Phi_{\lambda_k}(u_k)}$. Then, $u_k \rightarrow u \in W_0^{1,p}(\Omega)$.	
	\end{cor}
\begin{proof}
Arguing as in the proof of  Lemma \eqref{semicontinuidade} we deduce that inequality in \eqref{CC} is strict, that is 
$L>\Phi_{\lambda}(u)$ which contradicts our assumption. Hence $J$ must be empty, that is
	 		\[\lim_{k\to\infty}\int_{\Omega}|u_k|^{p^\star} dx= \int_{\Omega}|u|^{p^\star} dx\]
	 		and the uniform convexity of $L^{p^\star}(\Omega)$ implies that 		$
	 		u_{k}\to u\mbox{ in }L^{p^\star}(\Omega).$
By the fact that $\Phi_{\lambda_k}(u_k)\to \Phi_\lambda(u)$, it follows that $\hat M(\|u_k\|^p)\to \hat M(\|u\|^p)$ which ensures  our claim because of the strict monotonicity of $\hat M$.
\end{proof}

	\section{Existence Results}
In this section, we study the existence of global and local minimizers for the energy functional, as well as Mountain Pass solutions. The existence of minimizers will be guaranteed by  conditions \eqref{beta1} and \eqref{beta2}, while the existence of a Mountain Pass solution will be achieved under assumption  \eqref{beta3}.	 Throughout the sequel we will always assume \eqref{rho1} and \eqref{rho2}.

	\subsection{Global Minimizers}
	First we look for global minimizers. Consider the problem
		\begin{equation}\label{minimizationproblem}
		I_\lambda:=\inf \left\{\Phi_\lambda(u) \ : \ u \in  W_0^{1,p}(\Omega) \right\}
	\end{equation}

	\begin{pro}\label{existenceglobalminimizer}
		The infimum in problem \eqref{minimizationproblem} is attained by some $u_\lambda\in W_0^{1,p}(\Omega)$, under conditions \eqref{beta1} and \eqref{beta2}. If $\lambda>\lambda_0^\star$, then $I_\lambda<0$ and $u_\lambda\neq 0$. If $\lambda< \lambda_0^\star$, then $I_\lambda=0$ and $u_\lambda=0$. 	\end{pro}
	\begin{proof}
		Condition \eqref{rho2}, combined with the following inequality, gives us the coercivity of the energy functional:
		
		\begin{equation*}
			\begin{array}{lcl}
			\Phi_\lambda(u)	&\geq&\ds{\frac{1}{p}\hat{M}(\| u \|^p)-\frac{S}{p^\star}\| u \|^{p^\star}-\lambda\int_{\Omega}F(x,u(x))dx}\\
				&=&\ds{\|u\|^r\left[\frac{1}{p}\frac{\hat{M}(\|u\|^p)}{\left(\|u\|^p\right)^{\frac{r}{p}}}-\frac{S}{p^\star}\|u\|^{p^\star-r}-\lambda\int_{\Omega}\frac{F(x,u(x))}{\|u\|^r}dx\right]}.
			\end{array}
		\end{equation*}
	Conditions \eqref{beta1} and \eqref{beta2} give us the sequential weak lower semi continuity of the energy functional, as proven in Lemma \eqref{semicontinuidade}. Therefore, $I_\lambda$ is reached.
	In order to analyse the sign of $I_\lambda$, we resort to system \eqref{systemlambdat} and Definition \eqref{extremalparameterdefinition}.

	If $0 \leq \lambda < \lambda_0^\star$, then for all $u \in W_0^{1,p}(\Omega)\backslash\{0\}$, $\lambda< \lambda_0(u)$ and so $0=\psi_{\lambda_0(u),u}(t_0(u))\leq \psi_{\lambda_0(u),u}(1)< \psi_{\lambda,u}(1)=\Phi_\lambda(u)$. Since $\Phi_\lambda(0)=0$, we are done.
	
	If $\lambda > \lambda_0^\star$,  there exists $u \in W_0^{1,p}(\Omega)\backslash \{0\}$ such that $\lambda > \lambda_0(u)$. And thus $0=\psi_{\lambda_0(u),u}(t_0(u))>\psi_{\lambda,u}(t_0(u))=\Phi_\lambda(t_0(u)u)$.
	\end{proof}


\begin{pro}\label{existencesolutioncriticalcasenonzero}
	Assume condition \eqref{beta1} with strict inequality. Then, there exists $u_{\lambda_0^\star}\in W_0^{1,p}(\Omega)\setminus\{0\}$  such that 
$I_{\lambda_0^\star}=\Phi_{\lambda_0^\star}(u_{\lambda_0^\star})=0$. 
\end{pro}
\begin{proof}
Fix  a sequence $\lambda_k\downarrow \lambda_0^\star$. From Proposition \ref{existenceglobalminimizer}, for each $k$ we can find $u_k\in  W_0^{1,p}(\Omega)\backslash\{0\}$ such that $I_{\lambda_k}=\Phi_{\lambda_k}(u_k)<0$. Since $\lambda_k\downarrow \lambda_0^\star$ and assumptions \eqref{f1},\eqref{f2} hold true, the coercivity of $\Phi_{\lambda_0^\star}$ tells us that $\{u_k\}$ is bounded, and therefore we may assume that $u_k\rightharpoonup u_{\lambda_0^\star}$. From Lemma \ref{semicontinuidade} we obtain
	\begin{equation*}
	\Phi_{\lambda_0^\star}(u_{\lambda_0^\star})\le \liminf_{k\to \infty}\Phi_{\lambda_k}(u_k)\le 0.
	\end{equation*}

Since for all $u \in W_0^{1,p}(\Omega)\backslash\{0\}$ one has $ \lambda_0^\star\leq \lambda_0(u)$, there holds $0=\psi_{\lambda_0(u),u}(t_0(u))\leq \psi_{\lambda_0(u),u}(1)\leq \psi_{\lambda_0^\star,u}(1)=\Phi_{\lambda_0^\star}(u)$. Therefore,   $I_{\lambda_0^\star}=\Phi_{\lambda_0^\star}(u_{\lambda_0^\star})=0$.

Let us prove that $u_{\lambda_0^\star}\neq 0$.

Assume the contrary. Let $L>S\frac{p}{p^\star}$ such that 
$\inf_{t > 0}\frac{\hat{M}(t)}{t^\frac{p^\star}{p}} \geq L.$ Thus, 
	 	\begin{equation*}
\left(L-S\frac{p}{p^\star}\right)\|u_k\|^{p^\star}	\leq {\hat M(\|u_k\|^p})-\frac{p}{p^\star}\|u_k\|_{p^\star}^{p^\star}	\le \lambda_kp\int_\Omega{F(x,u_k)}dx.
	\end{equation*}
The right hand side tends to zero by the growth of $F$, and by using that $u_k\to 0$ in $L^q(\Omega)$ for $q<p^\star$.
	Hence, $u_k\to 0$ in  $W_0^{1,p}(\Omega)$. Dividing the previous inequality by $\|u_k\|^p$, we get 
	\begin{equation*}
	\frac{\hat M(\|u_k\|^p)}{\|u_k\|^p}-S\frac{p}{p^\star}\|u_k\|^{p^\star-p}\le \frac{\lambda_k}{\|u_k\|^p}p\int_\Omega{F(x,u_k)}dx.
	\end{equation*} And since the right hand side is still tending to zero, we get the desired contradiction, by \eqref{rho1}.
\end{proof}

{\bf Proof of Theorem  \ref{main1}.}
The existence of a global minimizer for $\Phi_\lambda$ follows by the coercivity of the energy functional and the lower semicontinuity property given by Lemma \ref{semicontinuidade}. The rest of the proof is a consequence of Propositions \ref{existenceglobalminimizer} and \ref{existencesolutioncriticalcasenonzero}.\qed

\subsection{Local Minimizers}
	We have already proved that the energy functional possesses a global minimizer $u_\lambda  \in W_0^{1,p}(\Omega)$, regardless of $\lambda$. However, when $0\leq \lambda < \lambda_0^\star$ Proposition \eqref{existenceglobalminimizer} tells us that $u_\lambda=0$, while Proposition \eqref{existencesolutioncriticalcasenonzero} states the existence of a non-trivial solution for $\lambda =\lambda_0^\star$. Therefore, for $0\leq \lambda <\lambda_0^\star$ we tackle a different minimization problem. Let $\delta>0$ and define
	
	\begin{equation}\label{localminimizationproblem}
		I_\lambda^\delta:=\inf_{u \in K_\delta}\Phi_{\lambda}(u),
	\end{equation}
where
\begin{equation*}
		K_\delta:=\left\{u \in W_0^{1,p}(\Omega)\ | \ d(u,K)\leq \delta \right\}
	\end{equation*}
 and
	\begin{equation*}
		K:=\left\{u \in W_0^{1,p}(\Omega)\backslash\{0\}  \ | \ \Phi_{\lambda_0^\star}(u)=0 \right\}.
	\end{equation*}
\begin{rem}\label{nice}
	Notice that by Proposition
	\eqref{existencesolutioncriticalcasenonzero} there holds $K\neq \emptyset$, as long as the inequality in condition \eqref{beta1} is strict.
\end{rem}
We are going to prove that problem \eqref{localminimizationproblem}   has  a  solution $u_\lambda^\delta$ which, for $0\leq \lambda < \lambda_0^\star$ close enough to $\lambda_0^\star$, does not belong to $\partial K_\delta$ and for $\delta $ small enough is non trivial. 

The proof of this fact is based on the following lemmas.

\begin{lem}\label{partialKdeltadoesnothavezero}
	Assume conditions  \eqref{f1}, and \eqref{f2} and  \eqref{beta1} with  strict inequality. Then, there exists $\bar\delta$ such that for  $0<\delta<\bar\delta$,   $0 \notin K_\delta$.	
\end{lem}
\begin{proof}
	Assume by contradiction that there exists $\delta_n\to 0$ with $d(0,K)\leq\delta_n$ for all $n \geq 1$. Therefore, $d(0,K)=0$, so $0 \in \bar{K}$. This implies the existence of   $\{u_n\}_{n \geq 1}\subset K$, $u_n \rightarrow 0 $ in $W_0^{1,p}(\Omega)$.  Thus, 
	\begin{equation*}
		\begin{array}{c}
			\ds{0=\|u_n\|^p\left[\frac{1}{p}\frac{\hat{M}(\|u_n\|^p)}{\|u_n\|^p}-\frac{1}{p^\star}\frac{\|u_n\|_{p^\star}^{p^\star}}{\|u_n\|^p}-\lambda_0^\star\int_{\Omega}\frac{F(x,u_n(x))}{\|u_n\|^p}dx\right]} \ \geq \\
			\ds{\|u_n\|^p\left[\frac{1}{p}\frac{\hat{M}(\|u_n\|^p)}{\|u_n\|^p}-\frac{1}{p^\star}S\|u_n\|^{p^\star-p}-\lambda_0^\star\int_{\Omega}\frac{F(x,u_n(x))}{\|u_n\|^p}dx\right]}
				\end{array}
	\end{equation*} and the latter is positive as it follows by 
	\eqref{rho1}, \eqref{f1}, and \eqref{f2}, leading to a contradiction.
\end{proof}

\begin{lem}\label{chalala} Assume conditions \eqref{f1}, \eqref{f2}, \eqref{beta2} and   \eqref{beta1} with  strict inequality. Then, there exists $\bar\delta$ such that for  $0<\delta<\bar\delta$, 	\begin{equation*}
		\inf_{u \in \partial K_\delta}\Phi_{\lambda_0^\star}(u)>0.
	\end{equation*}
\end{lem}
\begin{proof}Suppose by contradiction that for $\delta$ small enough there exists $\{u_k\}_{k \geq 1}\subset \partial K_\delta$ such that $\lim_{k \to \infty}\Phi_{\lambda_0^\star}(u_k)=0$. Then, by the coercivity of the energy functional, $u_k \rightharpoonup w$  in $W_0^{1,p}(\Omega)$ up to a subsequence. Also, by the weak lower semi-continuity of $\Phi_{\lambda_0^\star}$ it follows that $\Phi_{\lambda_0^\star}(w)=0$, which means that  $w \in K \cup \{0\}$ and that  $u_k \rightarrow w$  strongly in $W_0^{1,p}(\Omega)$  by Corollary \eqref{corolariosemicontinuidade}, i.e.  $w \in \partial K_\delta$. Subsequently, by Lemma \eqref{partialKdeltadoesnothavezero},   $w\neq 0$, and we reach the desired contradiction: $w \in K$ along with $d(w,K)=\delta>0$.
\end{proof}	
\begin{lem}\label{kcompact}
	 Assume conditions  \eqref{f1}, \eqref{f2}, \eqref{beta2}, and  \eqref{beta1} with strict inequality. Then, the set $K$ is compact in $W_0^{1,p}(\Omega)$ with the strong topology.
\end{lem}
\begin{proof}
	Let $\{v_n\}_{n \geq 1} \subset K$. The definition of $K$ and the coercivity of the energy functional imply that the sequence $\{v_n\}_{n \geq 1} \subset K$ is bounded in $W_0^{1,p}(\Omega)$. And consequently  $v_n\rightharpoonup v$ in $W_0^{1,p}(\Omega)$  up to a subsequence. By the sequential weak lower semi-continuity of the energy functional we obtain that $0\leq \Phi_{\lambda_0^\star}(v) \leq \liminf_{n \to \infty}\Phi_{\lambda_0^\star}(v_n)=0$, which implies that $v \in K \cup \{0\}$. Therefore, we obtain that $v_n \rightarrow v$  in $W_0^{1,p}(\Omega)$ by Corollary \eqref{corolariosemicontinuidade}. Since $K\subset K_\delta$ for all $\delta>0$, Lemma \eqref{partialKdeltadoesnothavezero} finishes the proof.
\end{proof}

\begin{lem}\label{kdeltaweaklyclosed}
	Assume conditions  \eqref{f1}, \eqref{f2}, \eqref{beta2} and condition \eqref{beta1} with  strict inequality. Then, the set $K_\delta$ is sequentially weakly closed in $W_0^{1,p}(\Omega)$.	
\end{lem}
\begin{proof}Let $\{u_n\}_{n \geq 1}\subset K_\delta$ be such that $u_n \rightharpoonup u_0$  in $W_0^{1,p}(\Omega)$. Since $d(u_n,K)\leq \delta$ for all $n\geq 1$, we will obtain that $d(u_0,K)\leq \delta$ if we show that $d(u_0,K)\leq \liminf_{n}d(u_n,K)$.
	
	To accomplish so, assume there exists a positive constant $c$ such that
	\begin{equation*}
		d(u_0, K)>c>\liminf_{n \to \infty}d(u_n,K).
	\end{equation*}
	By Lemma \eqref{kcompact}, $K$ is a compact subset of $\left(W_0^{1,p}(\Omega),\| \cdot\|\right)$. Therefore, for each $n$, there exists  $v_n \in K$ such that $d(u_n,K)=\|u_n-v_n\|$. Up to a subsequence, $v_n \rightarrow v_0 \in K$. The following chain of inequalities 	leads us to a contradiction: 	\begin{equation*}
		\begin{array}{rrl}
			\|u_0-v_0\| &\leq&\liminf_{n \to \infty} \|u_n-v_0\|\\& \leq &	\liminf_{n \to \infty}  (\|u_n-v_n\|+\|v_n-v_0\|) \\
			& =    &	\liminf_{n \to \infty}  (d(u_n,K)+ \|v_n-v_0\| )\\
			& <    & 	c+\lim_{n \to \infty}\|v_n-v_0\| \\
			& <    &	 d(u_0,K)	
		\end{array}
	\end{equation*}
\end{proof}

	{\bf Proof of Theorem \ref{main2}}. 
		Let $\delta>0$ small enough  be such that $0 \notin K_\delta $ (see Lemma \eqref{partialKdeltadoesnothavezero}. By \eqref{localminimizationproblem}, for $\lambda\leq \lambda_0^\star$ there exits a sequence $\{u_n\}_{n \geq 1} \subset K_\delta$ such that
		\begin{equation*}\label{sequenciaminimizantkdelta}
			\lim_{n \to \infty}\Phi_\lambda(u_n)=I_\lambda^\delta.
		\end{equation*}
	Since the functional $\Phi_\lambda$ is coercive, $\{u_n\}_{n \geq 1}$ must be bounded on $W_0^{1,p}(\Omega)$. And therefore up to a subsequence
		\begin{equation*}
			u_n \rightharpoonup u_\lambda^\delta \in W_0^{1,p}(\Omega).
		\end{equation*}
		By Lemma \eqref{kdeltaweaklyclosed}, $u_\lambda^\delta \in K_\delta$ and so 
		\begin{equation*}\label{sequenciaminimizantkdelta}
		I_\lambda^\delta:=	\inf_{u \in K_\delta}\Phi_{\lambda}(u)=\Phi_\lambda(u_\lambda^\delta ).
		\end{equation*}
		Let us show that if we take the parameter $\lambda\leq \lambda_0^{\star}$  close enough to $\lambda_0^\star$,  $u_\lambda^\delta$ does not belong to  $\partial K_\delta$.		
		Otherwise, 
		there exists  a sequence of positive numbers $\{\lambda_k\}_{k \geq 1}$, with $\lambda_k\leq \lambda_0^{\star}$, and $\lim_{k \to \infty}\lambda_k=\lambda_0^\star$, such that $u_{\lambda_k}^\delta \in \partial K_\delta$ for each $k$. Since $\partial K_\delta \subset K_\delta$, for all  $ k \geq 1$,
		\begin{equation*}
			I_{\lambda_k}^\delta=\inf_{u \in \partial K_\delta}\Phi_{\lambda_k}(u).
		\end{equation*}
	Fix any $u_0 \in K$. From the considerations above, we obtain that for all $k \geq 1$,
	\begin{equation*}
		 \inf_{u \in \partial K_\delta}\Phi_{\lambda_0^{\star}}(u)\leq \inf_{u \in \partial K_\delta}\Phi_{\lambda_k}(u)= \inf_{u \in  K_\delta}\Phi_{\lambda_k}(u)
		 \leq\inf_{u \in  K}\Phi_{\lambda_k}(u)\leq \Phi_{\lambda_k}(u_0).
	\end{equation*}
Since $ \Phi_{\lambda_k}(u_0)\rightarrow \Phi_{\lambda_0^\star}(u_0)=0$, and $\Phi_{\lambda_0^\star}\geq 0$ we obtain that
	        \begin{equation*}
	        	\inf_{u \in \partial K_\delta}\Phi_{\lambda_0^{\star}}(u)=0,
	        \end{equation*}
        which contradicts Lemma \eqref{chalala}. Thus, $u_\lambda^\delta \notin \partial K_\delta$ and so it is a local minimizer of $\Phi_\lambda$. 

        Also,  by Lemma \eqref{partialKdeltadoesnothavezero} $u_\lambda^\delta\neq 0$,  and by Remark \eqref{eitaa} there holds $I_\lambda^\delta>0$.
	\qed
	
	\subsection{Mountain Pass Solutions}
	In this section we prove the Palais-Smale property  for the energy functional and our  existence result will be a consequence of the Mountain Pass theorem.  
	
	First, we establish the Palais-Smale condition. We follow the proof of \cite[Theorem 1.3 ]{faraci2020critical}. 
		\begin{lem}\label{PS}
		Assume condition \eqref{beta3}. Then, for any $\lambda\geq0$ the energy functional $\Phi_\lambda$  satisfies the Palais-Smale property.
	\end{lem}
	\begin{proof}
		
		Let $\{u_k\}_{k \geq 1}$ be a Palais Smale sequence at a level $c$, i.e. a sequence satisfying the conditions
		\begin{equation*}
			\left\{\begin{array}{lll}
				\ds{\lim_{k \to \infty}\Phi_\lambda(u_k)}&=&c\\
				\ds{\lim_{k \to \infty}\Phi^{'}_\lambda(u_k)}&=&0.
			\end{array}\right.
		\end{equation*}
		
		Since \begin{equation*}
			\ds{\inf_{t > 0}\frac{M(t)}{t^{\frac{p^\star}{p}-1}}>S},
		\end{equation*}
		we may assume there exists a $L>S$ such that
		
		\begin{equation}\label{aab}
			\begin{array}{lr}
			M(t)>Lt^{\frac{p^\star}{p}-1}&\forall t\geq 0.
			\end{array}
		\end{equation}
		Then,
		\begin{equation*}
			\begin{array}{lr}
			\hat{M}(t)\geq L\frac{p}{p^\star}t^{\frac{p^\star}{p}}&\forall t\geq 0.
			\end{array}
		\end{equation*}

		In order to prove the coercivity of the energy functional, we may argue as in the proof of Proposition \eqref{existenceglobalminimizer}. The sequence $\{u_k\}_{k \geq 1}$ is bounded in $W_0^{1,p}(\Omega)$. And, up to subsequences, the following holds true.
		
		\begin{equation*}
			\left\{\begin{array}{lccr}
				u_k \rightharpoonup u&in&W_0^{1,p}(\Omega)& \\
				u_k\rightarrow u&in&L^q(\Omega),&q\in [1,+\infty) \\
				u_k\rightarrow u&a.e \ on&\Omega .&
			\end{array}\right.
		\end{equation*}
		By  the Concentration-Compactness Lemma of Lions
		there exist an at most countable index set $J$,
	 		a set of points $\{x_{j}\}_{j\in J}\subset\overline\Omega$ and two families of positive
	 		numbers $\{\eta_{j}\}_{j\in J}$, $\{\nu_{j}\}_{j\in J}$ such that
	 		\begin{align*}
	 		|\nabla u_{k}|^{p} & \rightharpoonup d\eta\geq|\nabla u|^{p}+\sum_{j\in J}\mu_{j}\delta_{x_{j}},\\
	 		|u_{k}|^{p^*} & \rightharpoonup d\nu=|u|^{p^*}+\sum_{j\in J}\nu_{j}\delta_{x_{j}},
	 		\end{align*}
	 		(weak star convergence in the sense of measures), where $\delta_{x_{j}}$ is the Dirac mass concentrated at
	 		$x_{j}$ and such that
	 		$$		S^{-\frac{p}{p^\star}}  \nu_{j}^{\frac{p}{p^*}}\leq\mu_{j} \qquad \mbox{for every $j\in J$}.$$
We will prove that $J$ is empty. 	
		Assume by contradiction that there exists an index $j_0 \in J$. And, for $\epsilon>0$ define the following smooth function on $\Omega$.

		\begin{equation*}
			\phi_\epsilon(x)=\left\{\begin{array}{crc}
				1, &x\in& B(x_0,\epsilon)\\
				0, &x \in&\Omega \backslash B(x_0,2\epsilon)
			\end{array}\right.
		\end{equation*}
		such that $				|\nabla \phi_\epsilon(x)|\leq \frac{2}{\epsilon}.$
		For each $\epsilon>0$, $\{u_k\phi_\epsilon\}_{k \geq 1}$ is bounded in $W_0^{1,p}(\Omega)$. Therefore,
		\begin{equation*}
			\lim_{k \to +\infty}\Phi_\epsilon^{'}(u_k)(u_k\phi_\epsilon)=0.
		\end{equation*}
		And thus
		\begin{equation}\label{aabb}\begin{array}{lll}
				o(1)&=&M(\|u_k\|^p)\ds{\int_{\Omega}|\nabla u_k|^{p-2}\nabla u_k(\nabla u_k \phi_\epsilon)dx-\int_{\Omega}|u_k|^{p^\star}\phi_\epsilon dx}-\lambda\ds{\int_{\Omega}f(x,u_k)u_k\phi_\epsilon dx}\\
				&=&M(\|u_k\|^p)\ds{\left[\int_{\Omega}|\nabla u_k|^p\phi_\epsilon+u_k|\nabla u_k|^{p-2}\nabla u_k \nabla \phi_\epsilon dx\right]-\int_{\Omega}|u_k|^{p^\star}\phi_\epsilon dx}-\lambda\ds{\int_{\Omega}f(x,u_k)u_k\phi_\epsilon dx}
			\end{array}
		\end{equation}

On the other hand,	by applying the Lebesgue Dominated Convergence Theorem, we may prove that
		\begin{equation*}\label{aa}
				\ds{\lim_{\epsilon \to 0}\lim_{k \to \infty}\int_{\Omega}u_k|\nabla u_k|^{p-2}\nabla u_k \nabla \phi_\epsilon dx}=0,
		\end{equation*}
	\begin{equation*}\label{aabbc}
		\begin{array}{lcl}
			\ds{\lim_{\epsilon \to 0}\lim_{k \to \infty}\int_{\Omega}|u_k|^{p^\star}\phi_\epsilon dx}=\ds{\lim_{\epsilon \to 0}\int_{B(x_0,2\epsilon)}|u_k|^{p^\star}\phi_\epsilon dx +\nu_{j_0}} 			=\nu_{j_0},
			
		\end{array}
	\end{equation*}
and, by condition \eqref{f1}, that

	\begin{equation*}\label{biel}
		\lim_{\epsilon \to 0}\lim_{k\to +\infty}\ds{\int_{\Omega}f(x,u_k)u_k\phi_\epsilon dx}=0.
	\end{equation*}
Since $M(\|u_k\|^p)$ is bounded in $\R$, we deduce that 
		
		\begin{equation*}\label{aabbccd}
			\ds{\lim_{\epsilon \to 0}\lim_{k \to \infty}M(\|u_k\|^p)\int_{\Omega}u_k|\nabla u_k|^{p-2}\nabla u_k \nabla \phi_\epsilon dx}=0.
		\end{equation*}
		
		
				
	
	

Furthermore, we have that 
\begin{equation*}\label{aabbcc}
			\begin{array}{lcl}
				\ds{ \lim_{k \to \infty} M(\|u_k\|^p)\int_{\Omega}|\nabla u_k|^p\phi_\epsilon dx }& \geq & \ds{\lim_{k \to \infty}\left[  M\left(\int_{\Omega}|\nabla u_k|^p dx\right)\int_{B(x_0,2\epsilon)}|\nabla u_k|^p\phi_\epsilon dx \right] }\\
				&\geq&\ds{\lim_{k \to \infty} \left[L \left(\int_{B(x_0,2\epsilon)}|\nabla u_k|^p\phi_\epsilon dx\right)^{\frac{p^\star}{p}-1}\int_{B(x_0,2\epsilon)}|\nabla u_k|^p\phi_\epsilon dx  \right]}\\
				&\geq&L\left[\ds{\int_{B(x_0,2\epsilon)}|\nabla u|^p\phi_\epsilon dx +\mu_{j_0}}\right]^{\frac{p^\star}{p}}.
			\end{array}
		\end{equation*}
		The above outcomes, combined together give us that

		\begin{equation*}
			\begin{array}{lcl}
				0\geq L\mu_{j_0}^{\frac{p^\star}{p}}-\nu_{j_0}
				=\left(L-S\right)\mu_{j_0}^{\frac{p^\star}{p}}
				\geq0,
			\end{array}
		\end{equation*}
		which means that $\mu_{j_0}=0$ and, subsequently that $\nu_{j_0}=0$, a contradiction. Thus, $J=\emptyset$ and so 
		\begin{equation}\label{layla}
			\ds{\lim_{k \to \infty}\int_{\Omega}|u_k|^{p^\star}dx=\int_{\Omega}|u|^{p^\star}dx},
		\end{equation}
		which implies that  $\{u_k\}$ converges to $u \in L^{p^\star}$ strongly.
		
		Let us  finally prove that $u_k\to u$ strongly in $W_0^{1,p}(\Omega)$. We already know by hypothesis that
		\begin{equation*}\label{aaab}
			\begin{array}{lcl}
				\ds 0={\lim_{k \to \infty}\Phi^{'}_\lambda(u_k)(u_k-u)}&=&\ds{\lim_{k \to \infty}\left[M(\|u_k\|^p)\ds{\int_{\Omega}|\nabla u_k|^{p-2}\nabla u_k\nabla (u_k-u)}dx-\int_{\Omega}|u_k|^{p^\star-2}u_k(u_k-u)dx\right.}\\
				&&\left.-\lambda\ds{\int_{\Omega}f(x,u_k)(u_k-u)dx}\right],
			\end{array}
		\end{equation*}
		

and, by using  condition \eqref{f1} and \eqref{layla} we obtain that
\begin{equation*}\label{bibi}
	\ds{\lim_{k\to +\infty}\int_{\Omega}f(x,u_k)(u_k-u)dx=0},
\end{equation*}

\begin{equation*}\label{hum}
	\ds{\lim_{k \to \infty}\int_{\Omega}|u_k|^{p^\star-2}u_k(u_k-u)dx}=0.
\end{equation*}
	
	Thus, 	
		\begin{equation*}\label{aaabc}
			\ds{\lim_{k \to \infty}M(\|u_k\|^p)\int_{\Omega}|\nabla u_k|^{p-2}\nabla u_k\nabla(u_k-u)dx}=0.
		\end{equation*}
		
		If $\ds{\lim_{k \to \infty}M(\|u_k\|^p)}=0$, then  from \eqref{aab} one has that $\ds{\lim_{k \to \infty}\|u_k\|=0 }$, which means that $u_k\rightarrow 0$ strongly  in $ W_0^{1,p}(\Omega)$. Otherwise, $\ds{\limsup_{k \to \infty}M(\|u_k\|^p)}>0$, that implies
		\begin{equation*}\label{aaabb}
			\ds{\lim_{k \to \infty}\int_{\Omega}|\nabla u_k|^{p-2}\nabla u_k\nabla(u_k-u)dx=0}.
		\end{equation*}
		Since $\{u_k\}_{k \geq 1}$ converges weakly to $u$, we know that
		\begin{equation*}\label{aaabbc}
			\ds{\lim_{k \to \infty}\int_{\Omega}|\nabla u|^{p-2}\nabla u\nabla(u_k-u)dx=0}.
		\end{equation*}
	Then, we deduce that		
		\begin{equation*}
			\ds{\lim_{k \to \infty}\int_{\Omega}\left(|\nabla u_k|^{p-2}\nabla u_k-|\nabla u|^{p-2}\nabla u\right)\nabla(u_k-u)dx=0},
		\end{equation*}
		and as a consequence, we obtain that $u_k \rightarrow u$ in  $W_0^{1,p}(\Omega)$.
		\end{proof}
	Let us point out now that the energy functional complies with the mountain pass geometry.  Notice that when $\lambda\geq \lambda_0^\star$ this is obvious since $0$ is a strict local minimizer (see Lemma \ref{geometry}) and $u_\lambda$ is a global minimizer with energy level less or equal than zero. Thus, the interesting case is when  $\lambda_0^\star-\epsilon<\lambda<\lambda_0^\star$.
	
\begin{lem}\label{geometry}
		 Under condition \eqref{beta3} the following statements hold true.

	\begin{itemize}
	
		\item[(i)] For $R>0$ small enough there exists  $\sigma=\sigma(R)>0$ such that $\Phi_{\lambda}(u)\geq \sigma$ for all $u \in W_0^{1,p}(\Omega)$ with $\|u\|=R$;
		
		\item[(ii)] For $0<R<\|u_{\lambda_0^*}\|$ in item $(i)$, there exists $\epsilon>0$ small enough such that $\Phi_{\lambda}(u_{\lambda_0^*})< \sigma$ for all $\lambda>\lambda_0^{\star}-\epsilon$.
	\end{itemize}
\end{lem}
\begin{proof}
	$(i)$
	By \eqref{rho1} there exists $C>0$ and $\delta>0$ such that if $\|u\|<\delta$ then
	\begin{equation*}
		\hat{M}(\|u\|^p)>C \|u\|^p.
	\end{equation*}
Also, from \eqref{f2} we deduce that for fixed  $\epsilon>0$ we may choose $\delta>0$ in such a way that $\ds{\int_{\Omega}F(x,u)dx}\leq \epsilon \|u\|^p$ for $\|u\|<\delta$.

Consequently, the following inequality holds true:
\begin{equation*}
		\Phi_{\lambda}(u)\geq \left(\frac{C}{p}-\lambda \epsilon\right)\|u\|^p-\frac{1}{p^\star}\|u\|_{p^\star}^{p^\star}\geq \|u\|^p\left[\left(\frac{C}{p}-\lambda \epsilon\right)-\frac{S}{p^\star}\|u\|^{p^\star-p}\right].
	\end{equation*}
Therefore, if we take  $\epsilon>0$ such  that $\frac{C}{p}-\lambda \epsilon>0$, and take $R< \delta=\delta(\epsilon)$ small enough so that, for all $u \in W_0^{1,p}(\Omega)$ such that $\|u\|=R$, there holds $\frac{S}{p^\star}\|u\|^{p^\star-p}<\frac{C}{p}-\lambda \epsilon$, we are led into the desired conclusion.

$(ii)$ From Proposition \eqref{existenceglobalminimizer},  $\Phi_{\lambda_0^\star}(u_{\lambda_0^\star})=0$. Therefore, the monotonicity of the function  $\lambda \mapsto \Phi_{\lambda}(u_{\lambda_0^\star})$ ensures  that $\Phi_{\lambda}(u_{\lambda_0^\star})\leq0$ for all $\lambda\geq\lambda_0^\star$, and that $0<\Phi_\lambda(u_{\lambda_0^{\star}})<\sigma$ for $\lambda_0^\star-\epsilon<\lambda<\lambda_0^\star$ with $\epsilon>0$ small enough.	
\end{proof}	

{\bf Proof of Theorem \ref{main3}.} Define
\begin{equation*}
	c_\lambda=\inf_{g \in \Gamma} \ \max_{0\leq t \leq 1}\Phi_{\lambda_0^\star}(g(t)),
\end{equation*}
where
\begin{equation*}
	\Gamma:=\left\{g \in C\left( [0,1];W_0^{1,p}(\Omega)\right) \ | \ g(0)=0, g(1)=u_{\lambda_0^\star}\right\}.
\end{equation*}
By Lemmas \ref{PS}, \ref{geometry} and  the  Mountain Pass Theorem it follows  that 
for each  $\lambda > \lambda_0^\star-\epsilon$, the set
	\begin{equation*}
		K_{c_\lambda}=\left\{u \in W_0^{1,p}(\Omega) \ | \ \Phi_{\lambda}(u)=c_\lambda, \ \Phi^{'}_\lambda(u)=0\right\}
	\end{equation*}
is not empty.
\qed
	
\section{Non-Existence Results}
In this section we establish a non-existence result under condition \eqref{beta3}, which is a stronger hypothesis than $\eqref{beta1}$. We also assume that $M $ is of class $C^1(\R)$, and that $f(x,\cdot)\in C^1(\R)$ for all $x\in \Omega$. We recall that hypotheses \eqref{rho1} and \eqref{rho2}  hold true. 

For each $u \in W_0^{1,p}(\Omega)\setminus\{0\}$ consider the following system:
\begin{equation}\label{systemlambdatderivative}\left\{\begin{array}{l}
		\psi^{'}_{\lambda,u}(t)=0\\
		\psi^{''}_{\lambda,u}(t)= 0\\
		\psi^{'}_{\lambda,u}(t)=\inf_{s > 0}\psi^{'}_{\lambda,u}(s).
	\end{array}\right.
\end{equation}

\begin{lem} Assume conditions \eqref{f1}, \eqref{f2} and  \eqref{beta3} .
	Then, for each $u \in W_0^{1,p}(\Omega)\backslash\{0\}$ system \eqref{systemlambdatderivative} possesses a  solution $(\lambda_1(u),t_1(u))$ which is unique with respect to $\lambda$. 
\end{lem}
\begin{proof}
	The proof is similar to Lemma \eqref{systemlambdatsolutionexistence}.
\end{proof}

\begin{lem}\label{lambdaonelessthanlambdazero}
Assume conditions \eqref{f1}, \eqref{f2} and  \eqref{beta3}. Then, for all $u \in W_0^{1,p}(\Omega)\backslash \{0\}$ there holds $\lambda_1(u)<\lambda_0(u)$.
\end{lem}	
\begin{proof}
	Let us perform a proof by contradiction. Suppose there exists a $u \in W_0^{1,p}(\Omega)\backslash \{0\}$ such that $\lambda_0(u)\leq\lambda_1(u)$. Then, $\psi^{'}_{\lambda_0(u),u}(t)\geq\psi^{'}_{\lambda_1(u),u}(t)\geq0$ for all $t\geq0$. Therefore, $\psi_{\lambda_0(u),u}$ is non-decreasing over $[0,+\infty)$. However, by Lemma \ref{fibersfirstproperties} we know that $\psi_{\lambda_0(u),u}(t)>0$ for $t>0$ small enough, $\psi_{\lambda_0(u),u}(t_0)=0$ for some $t_0 \in (0,+\infty)$, and that $\lim_{t\to +\infty}\psi_{\lambda_0(u),u}(t)=+\infty$. The proof is complete.
\end{proof}

	Let us define the extremal parameter	
	\begin{equation*}
		\lambda^{\star}_1:=\inf_{u \in W_0^{1,p}(\Omega)\backslash\{0\}}\lambda_1(u).
	\end{equation*}

Notice that by Lemma \ref{lambdaonelessthanlambdazero}, $\lambda_1^\star\leq\lambda_0^\star$.

\begin{lem}
Assume conditions \eqref{f1}, \eqref{f2} 	and  \eqref{beta1} with  strict inequality. Then,  $\lambda_0^{\star}=\lambda_0(u_{\lambda_0^{\star}})$, where $u_{\lambda_0^{\star}}$ is as in  Proposition \eqref{existencesolutioncriticalcasenonzero}.
\end{lem}

\begin{proof}Let $u:=u_{\lambda_0^{\star}}$. From the definition of $\lambda_0(u)$ and $\lambda_0^\star$,
	\begin{equation*}
		\begin{array}{lrlccl}
0=&I_{\lambda_0^{\star}}&=&\Phi_{\lambda_0^\star}(u)&\geq&\Phi_{\lambda_0(u)}(u) \\
		&&=&\psi_{\lambda_0(u),u}(1)&\geq&\psi_{\lambda_0(u),u}(t_0(u))\\
		&&=&0.&&
		\end{array}
	\end{equation*}
Therefore, $\Phi_{\lambda_0^\star}(u)=\Phi_{\lambda_0(u)}(u)$ or             \begin{equation*}
             \begin{array}{c}
             \ds{\frac{1}{p}\hat{M} (\| u \|^p)-\frac{1}{p^\star}\| u \|^{p^\star}_{p^\star}-\lambda_0^\star\int_{\Omega}F(x,u(x))dx=\frac{1}{p}\hat{M} (\| u \|^p)-\frac{1}{p^\star}\| u \|^{p^\star}_{p^\star}-\lambda_0(u)\int_{\Omega}F(x,u(x))dx}.	
             \end{array}
         \end{equation*}		
 Therefore, $\lambda_0^\star=\lambda_0(u_{\lambda_0^\star})$.
\end{proof}

\begin{lem} Assume conditions \eqref{f1}, \eqref{f2} and  \eqref{beta3}. Then, 
$\lambda_1^{\star}<\lambda_0^{\star}$.
\end{lem}
\begin{proof}
	From Lemma \eqref{lambdaonelessthanlambdazero} we obtain the desired inequality:
\begin{equation*}\lambda_1^{\star}\leq\lambda_1(u_{\lambda_0^{\star}})< \lambda_0(u_{\lambda_0^\star})=\lambda_0^{\star}.
	\end{equation*}
\end{proof}

We are ready to prove our non existence result. 

{\bf Proof of Theorem \ref{main0}}
	Assume  $\lambda< \lambda_1^{\star}$. Then,  for all $u \in W_0^{1,p}\backslash\{0\}$, $\lambda<\lambda_1(u)$. Therefore, $\psi^{'}_{\lambda,u}(t)>\psi^{'}_{\lambda_1(u),u}(t)\geq \psi^{'}_{\lambda_1(u),u}(t_1(u))=0$ for all positive $t$, and therefore the energy functional has no non-zero critical points.
\qed

\bigskip

\bigskip

	{\bf Acknowledgments} G. N. Cunha has been supported by FAPEG, Funda\c c\~ ao de Amparo \`a Pesquisa do Estado de Goi\'as. F. Faraci has been supported by Universit\`{a} degli Studi di Catania, PIACERI 2020-2022,  Linea di intervento 2, Progetto "MAFANE" and by the Gruppo Nazionale per l'Analisi Matematica, la Probabilit\`{a}
	e le loro Applicazioni (GNAMPA) of the Istituto Nazionale di Alta Matematica (INdAM). K. Silva has been supported by CNPq-Grant 308501/2021-7.

\end{document}